\documentclass{amsart}
\usepackage{graphicx} 
\usepackage{amsthm}
\usepackage{amsmath}
\usepackage{amsfonts}
\usepackage{url}
\usepackage{tikz}
\usetikzlibrary{calc,intersections,through,backgrounds}
\usepackage{tkz-euclide}
\usepackage{tcolorbox}

\usepackage{makecell}

\theoremstyle{plain}
\newtheorem{thm}{Theorem}
\newtheorem{lem}{Lemma}
\newtheorem{prop}[thm]{Proposition}
\newtheorem{cor}{Corollary}

\theoremstyle{definition}

\newtheorem{conj}{Conjecture}

\title{Tiling Triangles With $2\pi/3$ Angles}
\author{Yan X Zhang}
\address{San Jos\'e State University}
\date{\today}

\begin{document}

\maketitle

\begin{abstract}
    Motivated by a question of Erd\"{o}s and inquiries by Beeson and Laczkovich, we explore the possible $N$ for which a triangle $T$ can tile into $N$ congruent copies of a triangle $R$. The \emph{reptile} cases (where $T$ is similar to $R$) and the \emph{commensurable-angles} cases (where all angles of $R$ are rational multiples of $\pi$) are well-understood. We tackle the most interesting remaining case, which is when $R$ contains an angle of $2\pi/3$ and when $T$ is one of $6$ ``sporadic'' specific triangles, of which only $2$ were known to have constructions.  For each of these, we create a family of constructions and conjecture that they are the only possible $N$ that occur for these triangles. 
\end{abstract}

\section{Introduction}

We say that a polygon $T$ \emph{tiles into} ($N$ copies of) polygon $R$ if $T$ is a disjoint union of congruent copies of $R$, in which case we call $R$ the \emph{tile}. Our motivating problem is

\begin{tcolorbox}
For which triples $(N,T,R)$ does a triangle $T$ tile into $N$ copies of a triangle $R$?
\end{tcolorbox}

This problem has been extensively studied by Beeson (\cite{beeson-seven}, \cite{beeson-equilateral}, \cite{beeson-isosceles}, \cite{beeson-triangletiling3}); there is also a significant overlap with  Laczkovich's program of understanding tilings of triangles $T$ into triangles \textbf{similar} to some $R$. Furthermore, this problem is a higher-resolution version of an open (\$25) Erd\"{o}s problem \cite{soifer2009}:

\begin{tcolorbox}
    For which $N$ does there exist a tiling of some triangle $T$ into some triangle $R$?
\end{tcolorbox}

Two well-studied special cases of our problems are:
\begin{enumerate}
    \item The problem is completely understood for the case of \emph{reptiling}  \cite{golomb1964replicating} (when $T$ and $R$ are similar). In this case the families $N = k^2$ (for all triangles) and $h^2+k^2$ and $3k^2$ (for right triangles) appear.
    \item The problem is well-understood if the tile $R$ has \emph{commensurable angles}\footnote{This notion is due to Laczkovich \cite{laczkovich1995tilings}, who gives a finite classification for the generalization where the tiles only need to be similar instead of congruent.}, meaning when all the angles of the tile are rational multiples of $\pi$. See e.g. \cite{beeson-seven} for what tilings arise; a few extra families of potential values of $N$ (specifically, $N=3k^2$ and $6k^2$) appear involving specific triangles. 
\end{enumerate}

For the rest of our paper, we assume that we have \emph{incommensurable angles} (that is, we do not have commensurable angles). The remaining cases take a finite number of forms (see \cite{beeson-seven}), each obeying some linear equation involving $\pi$. For each of these, the angles of the tiled triangle are also constrained. See Figure~\ref{fig:beeson-incommensurable} for a list.

\begin{figure}
    \centering
    \includegraphics[scale=0.8]{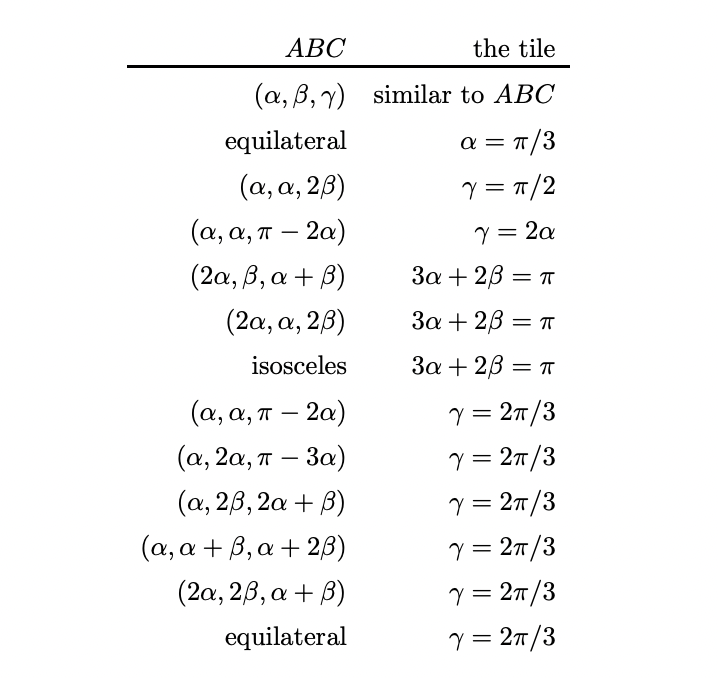}
    \caption{A List of the possible incommensurable-angles cases, where $ABC$ is the triangle being tiled. This table is taken (with permission) from Beeson \cite{beeson-seven}, though it contains the same content as Theorem 4.1 of \cite{laczkovich1995tilings}.}
    \label{fig:beeson-incommensurable}
\end{figure}

Very little is known in general about the incommensurable angles cases:
\begin{enumerate}
    \item Herdt (communicated through works of Beeson) has some specific constructions for the equilateral and isosceles triangles; they contain the smallest known\footnote{Our constructions in this paper independently reproduce Herdt's constructions, although we note that Herdt's constructions remain the smallest $N$ in their respective families.} $N$ for those cases with $R$ having an angle of $2\pi/3$, and one with an angle of $\pi/3$.  
    \item Beeson proved that $N=7$ is impossible \cite{beeson-seven} and no prime $N$ is possible in many of the cases (see \cite{beeson-seven}, \cite{beeson-equilateral}, \cite{beeson-isosceles}, \cite{beeson-triangletiling3}), sometimes ruling out bigger families of $N$ such as squarefree numbers. For the case when one of the angles of $R$ is $2\pi/3$ and $T$ is isosceles, Beeson \cite{beeson-isosceles} gives a lower bound of $2736$ for potential $N$.
    \item Laczkovich \cite{laczkovich2020rational} related possible $N$ for equilateral triangles to rationality of points on an elliptic curve.
\end{enumerate} 

In our work, we focus on the $\gamma = 2\pi/3$ case, which appears the most frequently ($6$ times) as subcases of Figure~\ref{fig:beeson-incommensurable}, and the closely related $\pi/3$ case. Our main contribution is proving the existence of $N$ for a large family of tilings for the incommensurable angles cases involving $2\pi/3$. This seems to be the first known construction for $3$ of those $6$ cases\footnote{In \cite{laczkovich1995tilings}, which applied to tilings by similar instead of congruent tiles, Laczkovich outlined how constructions would exist, but did not concretize them.}. We also conjecture that this construction solves the question of what $(N,T,R)$ are possible (when $R$ satisfies our conditions) for sufficiently big $N$; we prove our conjecture under certain assumptions for $T$ being an equilateral triangle.

In Section~\ref{sec:prelims} we give preliminary knowledge and assumptions, including our main tool \emph{ideal trapezoids}. In Section~\ref{sec:equilateral} we give our main construction and result in Theorem~\ref{thm:equilateral-construction-1}. In Section~\ref{sec:construction} we show how our construction can be used to produce families of tilings for all the incommensurable-angles cases is Figure~\ref{fig:beeson-incommensurable}. We end with some remarks in Section~\ref{sec:conclusion}. 

\section {Preliminaries}
\label{sec:prelims}

For this paper, unless specified, we assume the tile $R$ is a triangle with side lengths $(a,b,c)$ and corresponding angles $(\alpha, \beta, \gamma)$. Except for one subsection (where we extend our techniques to the $\gamma = \pi/3$ case), we will assume $\gamma=2\pi/3.$ 

\subsection{Integrality Assumptions}

If the pairwise ratios of $(a, b, c)$ are all rational, we say that the tile has \emph{commensurable sides} (this is sometimes just called \emph{commensurable} in the literature, but we emphasize "side" to avoid confusion with angles). A tile with commensurable sides can be scaled to be integers (in fact, pairwise coprime), turning the $c^2 = a^2 + b^2 + ab$ (by the law of cosines) relationship into a diophantine equation. \textbf{For the rest of this paper, we assume that $(a,b,c)$ are all integers.}

This seems to be a very strong restriction, but it ends up being necessary\footnote{In an older version of this manuscript we set this up as a conjecture, before it was subsequently proven.}. Our recent joint work with Beeson gave the following result, extending similar results in \cite{laczkovich2012tilings} and \cite{beeson-isosceles}: 

\begin{thm}\cite[Theorem 1.2]{rationality}
    \label{thm:rationality}
 Let triangle $T$ be tiled by a tile $R$ such that
    \begin{itemize}
        \item $R$ is not similar to $T$;
        \item $R$ is not a right triangle;
        \item $R$ has incommensurable angles.
    \end{itemize}
    Then $R$ must have commensurable sides.
\end{thm}

While this result is not necessary for this work (which is all constructive), it tells us that we are not losing any information by assuming integrality of the tile!

\subsection{Ideal Trapezoids}
\label{sec:trapezoids}

We say that trapezoid $ABCD$ is \emph{ideal} if $AB$ is parallel to $CD$ and $DAB = ABC = \pi/3$. In such a trapezoid, we use $x(ABCD)$ to denote the length of the shorter parallel side $|CD|$ and $y(ABCD)$ to denote the longer $|AB|$, and we abbreviate them as simply $x$ and $y$ when the context is clear. We first explore which trapezoids are tileable by $(a,b,c)$.

\begin{lem}
    \label{lem:trapezoid-1} Let $ABCD$ be an ideal trapezoid.  Then if $x = a^2 + b^2$ and $y = ab$, $ABCD$ can be tiled by $(a,b,c)$.
\end{lem}

\begin{figure}
    \centering
    \begin{tikzpicture}[scale=1.5]
        \tkzDefPoint(0,0){A}
        \tkzDefPoint(4,0){B}
        \tkzDefPoint(5,0){G}
\begin{scope}[shift=(A)]
\tkzDefPoint(60:15){X}
\tkzDefPoint(40:15){Y}
\end{scope}
\begin{scope}[shift=(B)]
\tkzDefPoint(120:15){X1}
\tkzDefPoint(160:15){Y2}
\end{scope}
\tkzInterLL(A,Y)(B,Y2)
\tkzGetPoint{E}
\begin{scope}[shift=(E)]
\tkzDefPoint(180:15){Z1}
\tkzDefPoint(0:15){Z2}
\end{scope}
\begin{scope}[shift=(G)]
\tkzDefPoint(120:15){G2}
\end{scope}
\tkzInterLL(E,Z1)(A,X)
\tkzGetPoint{D}
\tkzInterLL(E,Z2)(B,X1)
\tkzGetPoint{C}
\tkzInterLL(E,Z2)(G,G2)
\tkzGetPoint{F}
\tkzDrawPoints(A,B,C,D,E)
\tkzLabelPoints(A,B,C,D,E)
\tkzDrawPolySeg(A,D,E,A,B,E,C,B)
\tkzLabelSegments[color=black,above=4pt](D,E){$a^2$}
\tkzLabelSegments[color=black,above=4pt](C,E){$b^2$}
\tkzLabelSegments[color=black,left=4pt](A,D){$ab$}

\tkzMarkAngle[arc=l,size=0.2](E,A,D)
\tkzMarkAngle[arc=l,size=0.2](B,E,C)
\tkzMarkAngle[arc=l,size=0.2](E,B,A)

\end{tikzpicture}
    \caption{The basic ideal trapezoid. The marked angles are equal to $\alpha$.}
    \label{fig:ideal-trapezoid-basic}
\end{figure}
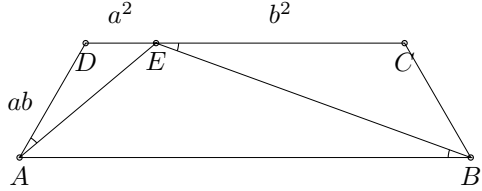

\begin{proof}
    See Figure~\ref{fig:ideal-trapezoid-basic}, specifically $ABCD$. We can split such a trapezoid into a $(a^2, ab, ac)$, a $(ba, b^2, bc)$, and a $(ca, cb, c^2)$ triangle, all of which are similar to the $(a,b,c)$ triangle by an integral factor and thus can be tiled by it.
\end{proof}

We use $\mathbb{N}_0$ to denote the set of nonnegative numbers $\{0, 1, 2, \ldots \}$. We will use the well-known result (e.g. \cite{sylvester-frobenius}):
\begin{prop}[Frobenius number of $2$ Elements]
\label{prop:frobenius-2}
If $\gcd(a,b) = 1$ and $x > ab-a-b$, then $x \in a\mathbb{N}_0 + b\mathbb{N}_0$; in other words, $x$ can be written as a nonnegative linear combination of $a$'s and $b$'s.
\end{prop}

\begin{lem}
    \label{lem:parallelogram-1} Let $Q$ by a parallelogram with angles $2\pi/3, \pi/3, 2\pi/3, \pi/3$ in clockwise order. Let the two side lengths be $x$ and $y$. then if $y = ab$ and $x > ab-a-b$, $Q$ can be tiled by $(a,b,c).$
\end{lem}

\begin{figure}
    \centering
\includegraphics[scale=0.4]{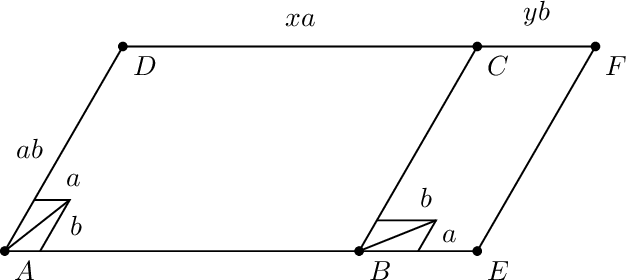}
    \caption{How to tile different parallelograms. The bigger parallelograms $ABCD$ and $BCFE$ are tiled by the same small $(a,b,a,b)$ parallelogram, which in turn tiles into two $(a,b,c)$ triangles each.}
    \label{fig:parallelograms}
\end{figure}

\begin{proof}
See Figure~\ref{fig:parallelograms} for the intuition. First, we can combine two $(a,b,c)$ triangles together to make a parallelogram $P$ with the same angles as in the assumption, with the sides in order $(a,b,a,b)$. Now, note that we are able to tile such a parallelogram with side lengths $(ab,a,ab,a)$ and also $(ab,b,ab,b)$ by two different orientations of $P$. By Proposition~\ref{prop:frobenius-2}, we are able to combine some nonnegative numbers of these two parallelograms to tile a $(ab,(ax+yb), ab, (ax+yb))$ parallelogram, where we can take $ax+yb$ to be any integer greater than $ab-a-b$. 
\end{proof}

\begin{prop}
    \label{prop:trapezoid-2} Let $ABCD$ be an ideal trapezoid.  Then if $x >  c^2-a-b$ and $(ab)|y$, $ABCD$ can be tiled by $(a,b,c)$.
\end{prop}

\begin{figure}
    \centering

\includegraphics[scale=0.35]{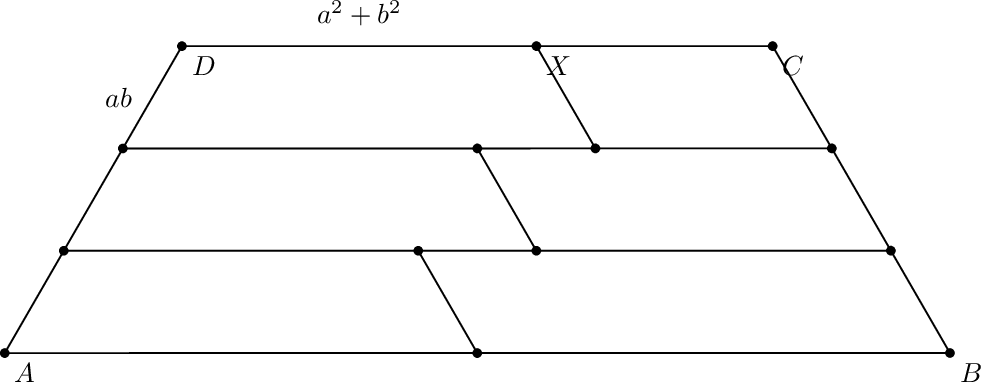}
    \caption{Tiling more complex ideal trapezoids.}
    \label{fig:ideal-trapezoid-complex}
\end{figure}

\begin{proof}
See Figure~\ref{fig:ideal-trapezoid-complex} for the intuition. First, (assuming angle $BAD$ and $CBA$ equal $\pi/3$), suppose $y = |AD| = kab$, where $k$ is an integer. Draw parallel lines that split $ABCD$ into $k$ ideal trapezoids where the lateral sides have length $ab$. We can split each of these into an ideal trapezoid with $x = a^2+b^2$ and $y=ab$, which we showed in Lemma~\ref{lem:trapezoid-1} to be tileable by $(a,b,c)$, and a parallelogram with one pair of sides having length $ab$. This divides the top ideal trapezoid (with one side $CD$) into two parts, one of length $|XD| = a^2+b^2$ and one of length $|CX|$, which we assumed to be greater than  
$$c^2-a-b - (a^2+b^2) = ab-a-b,$$
meaning that the parallelogram with side $CX$ is also tileable by $(a,b,c)$ according to Lemma~\ref{lem:parallelogram-1}. The other $(k-1)$ ideal trapezoids that $ABCD$ split into have longer side lengths, so Lemma~\ref{lem:parallelogram-1} also applies to them. Thus, $ABCD$ is tileable into $(a,b,c).$ 
\end{proof}

\section{Equilateral Triangles}
\label{sec:equilateral}

\subsection{An Infinite Family of Tilings}

If $(X,X,X)$ tiles into $(a,b,c)$, we say that $X$ is \emph{equiconstructible} by $(a,b,c)$. By comparing areas of $T$ and $R$, we see that the number of tiles is $$N := N(X,a,b,c) = \frac{\frac{1}{2}X^2 \sin{(\pi/3)}}{\frac{1}{2}ab \sin{(2\pi/3)}} = \frac{X^2}{ab}.$$



\begin{figure}
    \centering

\includegraphics[scale=0.5]{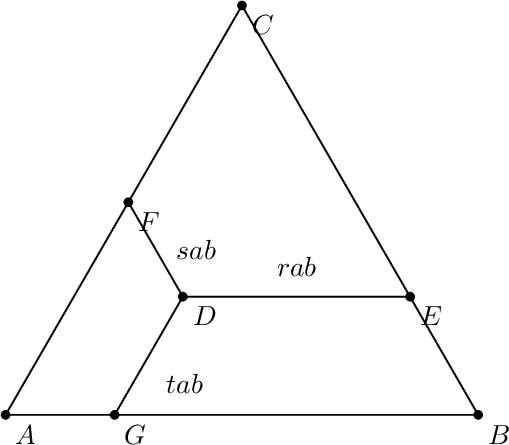}
    \caption{Tiling the equilateral triangle into $3$ ideal trapezoids.}
    \label{fig:equilateral-construction-1}
\end{figure}

\begin{thm}
    \label{thm:equilateral-construction-1} Let $M = 3\lceil \frac{c^2 - a - b}{ab} \rceil$. Then for all integers $m \geq M$, $mab$ is equiconstructible by $(a,b,c)$. As a consequence, there is a $m^2 ab$ tiling for all such $m$.
\end{thm}
\begin{proof}
Consider $3$ integers $r,s,t$ that are all at least $M$. Consider the ideal trapezoid with $x = rab$ and $y = sab$. Because $r$ is an integer and $s \geq M$, Proposition~\ref{prop:trapezoid-2} applies and this trapezoid can be tiled by $(a,b,c)$. By also doing this with $(x = sab, y = tab)$ and $(x = tab, y=rab)$, we obtain $3$ ideal trapezoids that can all be tiled by $(a,b,c)$. 

Now, an equilateral triangle with sides $(r+s+t)ab$ can be tiled into these $3$ trapezoids as in Figure~\ref{fig:equilateral-construction-1}. This means $(r+s+t)ab$ is equiconstructible, giving a tiling with $(r+s+t)^2ab$ tiles. By construction, $(r+s+t)$ can be taken to be any integer at least $M$, which finishes the proof.
\end{proof}

The consequence of Theorem~\ref{thm:equilateral-construction-1} is ``sharp'' for \emph{square-free} side lengths:
\begin{lem}
\label{lem:squarefree}
If $a$, $b$ are square-free, then if $X$ is equiconstructible, then we must have $X = mab$ for some integer $m$.
\end{lem}
\begin{proof}
    We need $X^2 = abN$, where $N$ is the number of tiles used. Since $\gcd(a,b) = 1$, $a$ and $b$ being square free implies $(ab)|N$ and thus $(ab)^2|X^2$. 
\end{proof}
\begin{conj}
\label{conj:equilateral}
All equiconstructible $X$ are divisible by $ab$. As the smallest interesting case, we conjecture that all equiconstructible $X$ for $(5,16,19)$ are divisible by $16$.
\end{conj}
Confirming this conjecture would resolve our motivating problem:
\begin{conj}
    \label{conj:formal-equilaeral}
    For all non-reptile and incommensurable tilings of $T$ into $R = (a,b,c)$ with an angle equal to $2\pi/3$, the possible $N$ is the set 
    $$\{m^2 ab | m \geq M\} $$
    where $M$ is defined as in Theorem~\ref{thm:equilateral-construction-1}.
\end{conj}

As an example, consider $(a,b,c) = (3,5,7)$. Then Theorem~\ref{thm:equilateral-construction-1} shows\footnote{The $m=9$ case corresponds to Herdt's known construction with side length $135$.} that $15m$ is equiconstructible for $m \geq 9$; furthermore, these are the only $X \geq 135$ that occur. Separate work by Beeson shows that $X < 105$ is known to not exist \cite{beeson-equilateral}, so this reduces understanding all equiconstructible $X$ for $(3,5,7)$ to only the cases $105$ and $120$, which we conjecture to not exist. 

\subsection{Tiles with a $\pi/3$ Angle}

We can obtain a similar result for another row in Figure~\ref{fig:beeson-incommensurable}, which is the case of incommensurable-angles where one of the angles is $\pi/3$ instead of $2\pi/3$. For this subsection only, we let $(a,b,c)$ be an integral-sided (again, \cite{laczkovich2012tilings} shows that this assumption actually loses nothing) tile with corresponding angles $(\alpha, \beta, \gamma = \pi/3).$ By the law of cosines, $c^2 = a^2 + b^2-ab$. 

\begin{lem}
    \label{lem:trapezoid-simple-alt} Let $ABCD$ be an ideal trapezoid.  Then if $x = c^2$ and $y = ab$, $ABCD$ can be tiled by $(a,b,c)$, where $c$ faces a $\pi/3$ angle.
\end{lem}

\begin{figure}
    \centering
\begin{tikzpicture}[scale=1.5]
        \tkzDefPoint(0,0){A}
        \tkzDefPoint(60:2){D}
        \tkzDefPoint(15,0){X}
\begin{scope}[shift=(D)]
\tkzDefPoint(290:15){D1}
\tkzDefPoint(0:15){D2}
\end{scope}
\tkzInterLL(A,X)(D,D1)
\tkzGetPoint{E}
\begin{scope}[shift=(E)]
\tkzDefPoint(50:15){E1}
\end{scope}
\tkzInterLL(E,E1)(D,D2)
\tkzGetPoint{C}
\begin{scope}[shift=(C)]
\tkzDefPoint(300:2){B}
\end{scope}
\tkzDrawPoints(A,B,C,D,E)
\tkzLabelPoints(A,B,C,D,E)
\tkzDrawPolySeg(D,A,E,D,C,E,B,C)
\tkzLabelSegments[color=black,above=4pt](A,E){$a^2$}
\tkzLabelSegments[color=black,above=4pt](B,E){$b^2$}
\tkzLabelSegments[color=black,left=4pt](A,D){$ab$}

\tkzMarkAngle[arc=l,size=0.2](A,D,E)
\tkzMarkAngle[arc=l,size=0.2](D,C,E)
\tkzMarkAngle[arc=l,size=0.2](B,E,C)

\end{tikzpicture}
    \caption{The basic ideal trapezoid, now tiled by $3$ similar triangles with angles $(\alpha, \beta, \pi/3)$. The marked angles are equal to $\alpha$.}
    \label{fig:ideal-trapezoid-basic-alt}
\end{figure}
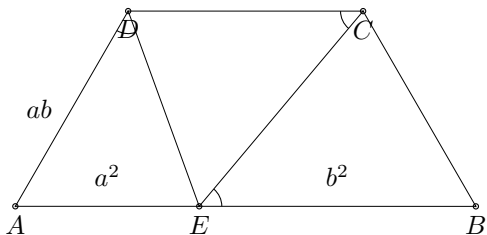

\begin{proof}
    See Figure~\ref{fig:ideal-trapezoid-basic-alt}. The idea is similar to that of Lemma~\ref{lem:trapezoid-1}, except that $|CD|$ is now $c^2$ and $|AB|$ is $a^2 + b^2$.
\end{proof}

\begin{thm}
    \label{thm:equilateral-construction-2} Let $M = 3\lceil \frac{a^2+b^2 - a - b}{ab} \rceil$. Then for all integers $m \geq M$, $mab$ is equiconstructible by $(a,b,c)$, where $c$ faces a $\pi/3$ angle. As a consequence, there is a $m^2 ab$ tiling for all such $m$.
\end{thm}
\begin{proof}
As in  Theorem~\ref{thm:equilateral-construction-1}, we combine two facts:
\begin{enumerate}
    \item By Lemma~\ref{lem:trapezoid-simple-alt}, the ideal trapezoid with $x=ab$ and $y=c^2$ is tileable by $(a,b,c).$ 
    \item The reasoning of Lemma~\ref{lem:parallelogram-1} still holds: we should be able to tile any parallelogram with angles $(\pi/3, 2\pi/3, \pi/3, 2\pi/3)$ with sides $(ab,k,ab,k)$ where $k > ab-a-b$, via parallelograms with sides $(a,b,a,b)$.
\end{enumerate} 
As a consequence, we can  tile an ideal trapezoid with $x=ab$ and 
$$y > c^2 + ab-a-b = (a^2 + b^2) - a - b$$
by combining a smaller ideal trapezoid and a parallelogram.

Therefore, suppose we have $3$ integers $r, s, t$ that are all at least $\lceil \frac{a^2+b^2 - a - b}{ab} \rceil$. As in Theorem~\ref{thm:equilateral-construction-1}, the ideal trapezoid with $x=rab$ and $y=sab$ is tileable by $(a,b,c)$, and we can combine $3$ of these to give a tiling of an equilateral triangle with side $(r+s+t)ab$ and $(r+s+t)^2ab$ tiles, as desired.
\end{proof}

Consider $R = (5,8,7)$, which has the $7$ facing a $\pi/3$ angle. Since $1< (5^2 + 8^2 - 8 - 5)/40 < 2$, our construction tiles an equilateral triangle of side $3*2*40 = 240$ with $1440$ tiles. This matches a construction by Herdt (private communication to Beeson), and we conjecture that it is the smallest such $N$ for such $(T,R)$.

\begin{conj}
Conjectures~\ref{conj:equilateral} and \ref{conj:formal-equilaeral} also hold in this setting.
\end{conj}

\section{The $(2\alpha, 2\beta, \alpha+\beta)$ Triangle}

The $(2\alpha, 2\beta, \alpha+\beta)$ triangle is another incommensurable-angles case of interest. By e.g. law of sines, the sides are in ratio $(a(b+2a), b(a+2b), c^2).$ We will approach this triangle in two ways.

\begin{figure}
    \centering
\includegraphics[scale=0.7]{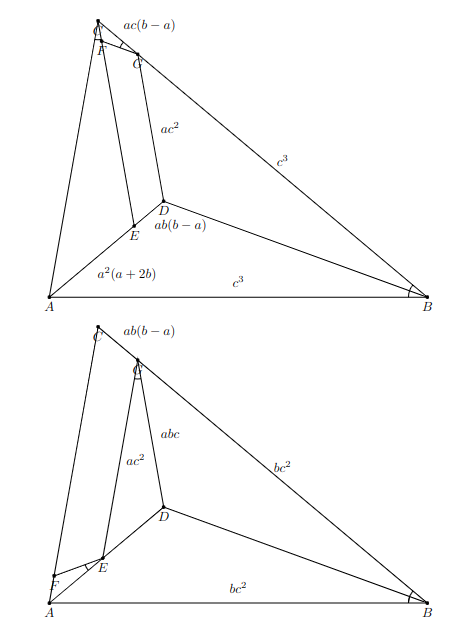}
    \caption{Triangles with angles $2\beta, 2\alpha, \alpha+\beta$. The marked angles are $\alpha$. The top figure has lengths $c/b$ times that of the bottom.}
    \label{fig:arithmetic-construction}
\end{figure}

\begin{prop}
\label{prop:arithmetic} For any integer $m \geq 1$, the following two tilings of triangles with angles $(2\alpha, 2\beta, \alpha+\beta)$ are possible:
\begin{enumerate}
    \item We can tile the triangle with sides $((a+2b)mac, (b+2a)mbc, c^3m)$ into $(b+2a)(a+2b)m^2c^2$ copies of $(a,b,c)$.
    \item We can tile the triangle with sides $((a+2b)mab, (b+2a)mb^2, bc^2m)$ into $(b+2a)(a+2b)m^2b^2$ copies of $(a,b,c)$.
\end{enumerate}
\end{prop}
\begin{proof}
For both parts, it suffices to prove the statement for $m=1$. See the top and bottom figures of Figure~\ref{fig:arithmetic-construction} respectively. In both cases, we can tile such a triangle into $4$ triangles similar to $(a,b,c)$ and a single ideal trapezoid with $x = ac^2$ and $y = ab(a-b)$. 

Since $ac^2 > c^2 - a - b$, Proposition~\ref{prop:trapezoid-2} applies and the trapezoid can be tiled by $(a,b,c)$. It is easy to check that the other lengths in the figure make the triangles tileable by $(a,b,c)$. 

For the top figure, since the area of $ABC$ is $\frac{1}{2}|AC||BC|\sin(C)$, we can compute that the number of total tiles is $$\frac{\frac{(a+2b)ac(b+2b)bc\sin(C)}{2}}{\frac{ab\sin(2\pi/3)}{2}} = (a+2b)(b+2a)c^2,$$
and scaling the sides by $m$ would scale the number of tiles by $m^2$. The bottom figure is similar, except with $c$ replaced by $b$ in the computation. 
\end{proof}

Using the second option of Proposition~\ref{prop:arithmetic} for $(3,5,7),$ we obtain:
\begin{cor}
\label{cor:arithmetic} There exists a tiling of a triangle with angles $(2\beta, 2\alpha, \alpha+\beta)$ into $3575$ copies of $(3,5,7)$.
\end{cor}

\begin{thm}
    \label{thm:arithmetic} For all integers $m > bc-c-b$, there is a $(b+2a)(a+2b)m^2$-tiling of some triangle with angles $(2\alpha, 2\beta, \alpha+\beta)$.
\end{thm}
\begin{proof}
    Consider, as in Figure~\ref{fig:arithmetic-tiling-combination}, a $(2\alpha, 2\beta, \alpha+\beta)$-angled triangle with longest edge $(c^2)(bk + ck')$. We can split this into $2$ triangles similar to the original triangle and then a parallelogram with angles $\pi/3$ and $2\pi/3$. Proposition~\ref{prop:arithmetic} shows the triangles can be tiled into copies of $(a,b,c)$. The parallelogram, having sides $b(b+2a)bk$ and $a(a+2b)ck'$ where the first side is divisible by $b$ and the second by $a$, can be tiled by $(a,b,c)$; specifically, by the parallelogram with sides $(a,b,a,b)$ and angles $(2\pi/3, \pi/3, 2\pi/3, \pi/3)$ obtained by putting two copies of $(a,b,c)$ together along $c$. 

Finally, we use Proposition~\ref{prop:frobenius-2} again: since $\gcd(b,c)=1$, all $m \geq bc-c-b$ can be written as some combination $bk + ck'$. This finishes the proof. 
\end{proof}

\begin{figure}
    \centering

\includegraphics[scale=0.3]{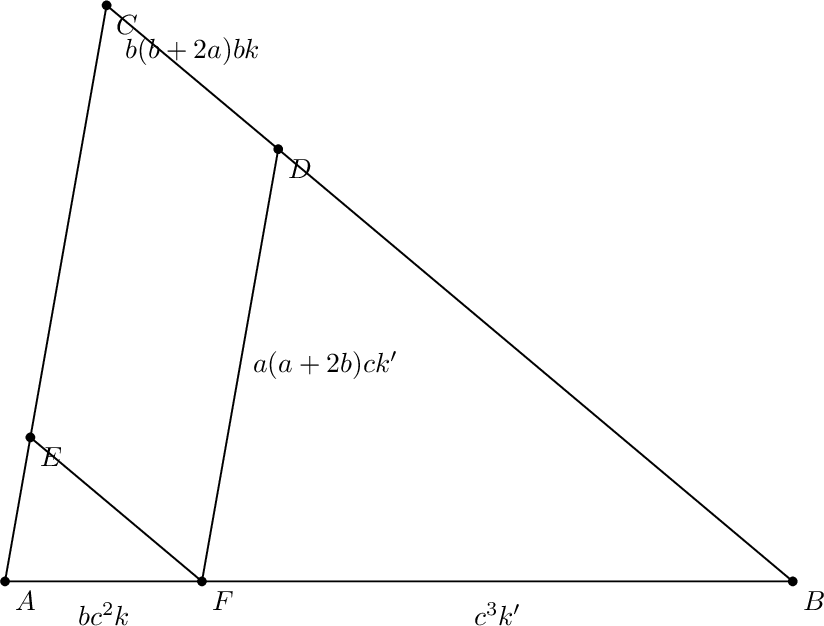}

    \caption{Combining two tilings to make another tiling with bottom side $c^2(bk+ck')$.}
    \label{fig:arithmetic-tiling-combination}
\end{figure}

This construction gives a similar consequence (and conjecture) as with Lemma~\ref{lem:squarefree}:

\begin{lem}
    \label{lem:squarefree-2} If $a \neq b \pmod{3}$, then if some triangle with angles $(2\alpha, 2\beta, \alpha+\beta)$ can be tiled into $N$ tiles, $N = (a+2b)(b+2a)m^2$ for some integer $m$. 
\end{lem}
\begin{proof}
    Let the sides be $(a+2b)am, (b+2a)bm, c^2m$, where all the sides are integers and $m \in \mathbb{Q}$ (not necessarily $\mathbb{N}$!). Notice that 
    $$\gcd(a, b) = \gcd(a, b+2a) = \gcd(b, a+2b) = 1.$$ 
    Now, $$\gcd(a+2b, b+2a) = \gcd(a+2b, b-a) = \gcd(a+2b, 3b) = \gcd(a+2b, 3).$$
    Since $a \neq b \pmod{3}$, this equals $1$, so $\gcd((a+2b)a, (b+2a)b) = 1$. Since $(a+2b)am$ and $(b+2a)bm$ are both integers, the coprimality we just showed deduces that $m$ must in fact be an integer as well. Thus, $m \in \mathbb{N}$, and the number of tiles is $(a+2b)(b+2a)m^2$.
\end{proof}

\begin{conj}
    For $(a,b,c)$, if there exists a tiling of a $(2\alpha, 2\beta, \alpha+\beta)$-angled triangle, the number of tiles must be of the form $(a+2b)(b+2a)m^2$. Any counterexample would require $a = b \pmod{3}.$
\end{conj}

\section{Other Constructions}
\label{sec:construction}

We have explored equilateral triangles and $(2\alpha, 2\beta, \alpha+\beta)$. In this section, we show that these constructions can be used to construct families for the other $4$ potential incommensurable-angles and non-reptile cases that tile into $(\alpha, \beta, \gamma = 2\pi/3)$ triangles. Recall that these have angles:
\begin{itemize}
    \item isosceles $(\alpha, \alpha, \pi-2\alpha)$; 
    \item $(\alpha, \alpha+\beta, \alpha+2\beta)$;
    \item $(\alpha, 2\beta, \beta+2\alpha)$;
    \item $(\alpha, 2\alpha, 3\beta)$;
\end{itemize}
(for the purpose of this list, each item includes the variation where $\alpha$ and $\beta$ are swapped; for example, the $(\alpha, \alpha, \pi-2\alpha)$ case also includes $(\beta, \beta, \pi-2\beta)$). For the first three we will use the equilateral triangle as an auxiliary tool. For the final case we will use the $(2\alpha, 2\beta, \alpha+\beta)$ triangle.

\begin{prop}
    \label{prop:isosceles} If $mab$ is equiconstructible by $(a,b,c)$, then:
    \begin{enumerate}
        \item we can tile $(mbc, mbc, mb(a+2b))$, which has angles $(\alpha, \alpha, \pi-2\alpha)$, into $m^2b(a+2b)$ copies of $(a,b,c)$.
        \item we can tile $(mac, mac, ma(b+2a))$, which has angles $(\beta, \beta, \pi-2\beta)$, into $m^2a(b+2a)$ copies of $(a,b,c)$.
    \end{enumerate}
\end{prop}

\begin{figure}
    \centering

\includegraphics[scale=0.4]{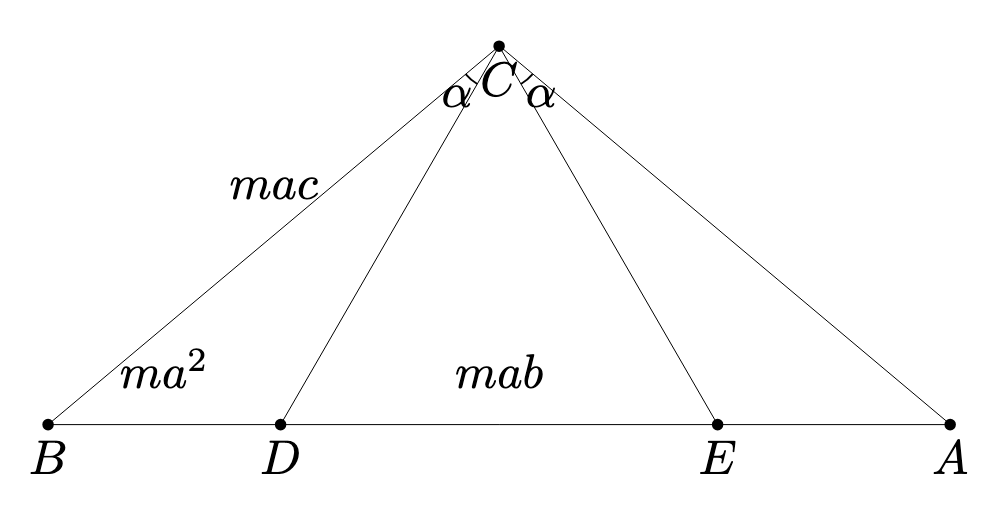}
    
    \caption{Isosceles triangle with sides $(mac, mac, ma(b+2a)).$}
    \label{fig:iscosceles-construction}
\end{figure}

\begin{proof}
    As in Figure~\ref{fig:iscosceles-construction}, we can tile such a triangle into 2 triangles similar to $(a,b,c)$ (with scaling factor $r$) and one equilateral triangle with side $rb$. Suppose $r = ma$, with $m$ an integer. Then $rb = mab$ is equiconstructible, and we can tile all three triangles by a total of $$m^2(ab) + 2m^2a^2 = m^2a(b+2a)$$
    copies of $(a,b,c)$. By replacing the roles of $a$ and $b$, we can get a similar tiling with $m^2b(a+2b)$ instead.
\end{proof}

\begin{prop}
    \label{prop:arithmetic-2-construction} If $mab$ is equiconstructible by $(a,b,c)$, then:
    \begin{enumerate}
        \item we can tile $(mab, mbc, mb(a+b))$, which has angles $(\alpha, \alpha+\beta, \alpha+2\beta)$, into $m^2b(a+b)$ copies of $(a,b,c)$.
        \item we can tile $(mab, mac, ma(a+b))$, which has angles $(\beta, \alpha+\beta, \beta+2\alpha)$, into $m^2a(a+b)$ copies of $(a,b,c)$.
    \end{enumerate}
\end{prop}
\begin{figure}
    \centering
        
\includegraphics[scale=0.3]{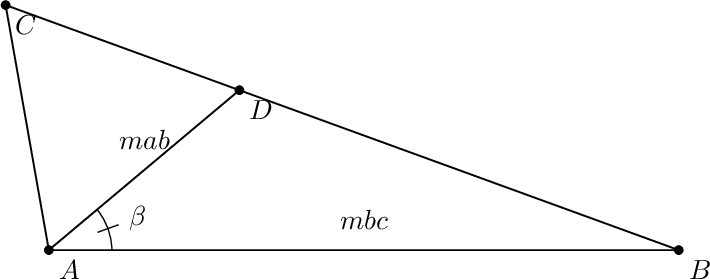}
    
    \caption{Triangle with angles $(\alpha, \alpha + \beta, \alpha + 2\beta)$.}
    \label{fig:arithmetic-2-construction}
\end{figure}

\begin{proof}
We do the first case (the second case is symmetric). As in Figure~\ref{fig:arithmetic-2-construction}, we can tile such a triangle into an equilateral triangle and a triangle similar to $(a,b,c)$. If we set $AD = mab$, this tiles the entire triangle into $m^2ab$ (from $ACD$) plus $m^2b^2$ (from $ABD$) copies of $(a,b,c),$ for a total of $m^2b(a+b)$.
\end{proof}

\begin{prop}
    \label{prop:odd-construction} If $mab$ is equiconstructible, then:
    \begin{enumerate}
        \item we can tile the triangle $(mac, (b+2a)mb, (a+b)cm)$ with angles $(\alpha, 2\beta, 2\alpha+\beta)$ into $m^2(b+2a)(a+b)$ copies of $(a,b,c)$;
        \item we can tile the triangle $(mbc, (a+2b)ma, (a+b)cm)$ with angles $(\beta, 2\alpha, \alpha+2\beta)$ into $m^2(a+2b)(a+b)$ copies of $(a,b,c)$;
    \end{enumerate} 
\end{prop}
\begin{figure}
    \centering

    \begin{tikzpicture}[scale=0.7]
\tkzDefPoint(0,0){A}
\tkzDefPoint(4,0){B}
\begin{scope}[shift=(A)]
\tkzDefPoint(85:15){X}
\end{scope}
\begin{scope}[shift=(B)]
\tkzDefPoint(110:15){Y}
\tkzDefPoint(145:15){Z}
\end{scope}
\tkzInterLL(A,X)(B,Y)
\tkzGetPoint{C}
\tkzInterLL(A,C)(B,Z)
\tkzGetPoint{D}
\tkzDrawPoints(A,B,C,D)
\tkzLabelPoints(A,B,C,D)
\tkzDrawPolySeg(D,A,B,D,C,B)
\tkzMarkAngle[arc=l,size=0.4](D,B,A)
\tkzLabelAngle[pos=0.6](D,B,A){$\beta$}
\tkzMarkAngle[arc=l,size=0.4](B,A,C)
\tkzLabelAngle[pos=0.6](B,A,C){$2\alpha+\beta$}
\tkzMarkAngle[arc=l,size=0.6](D,C,B)
\tkzLabelAngle[pos=0.9](D,C,B){$\alpha$}
\tkzMarkAngle[arc=l,size=0.6](C,B,D)
\tkzLabelAngle[pos=0.9](C,B,D){$\beta$}

\tkzLabelSegments[color=black,above=4pt](A,D){$mab$}
\tkzLabelSegments[color=black,above=4pt](C,D){$mb(a+b)$}
\tkzLabelSegments[color=black,above=4pt](B,D){$ma(a+b)$}
\tkzLabelSegments[color=black,below=4pt](A,B){$mac$}
\end{tikzpicture} 

    \caption{Triangle with angles $(\alpha, 2\beta, 2\alpha+\beta$).}
    \label{fig:odd-construction}
\end{figure}
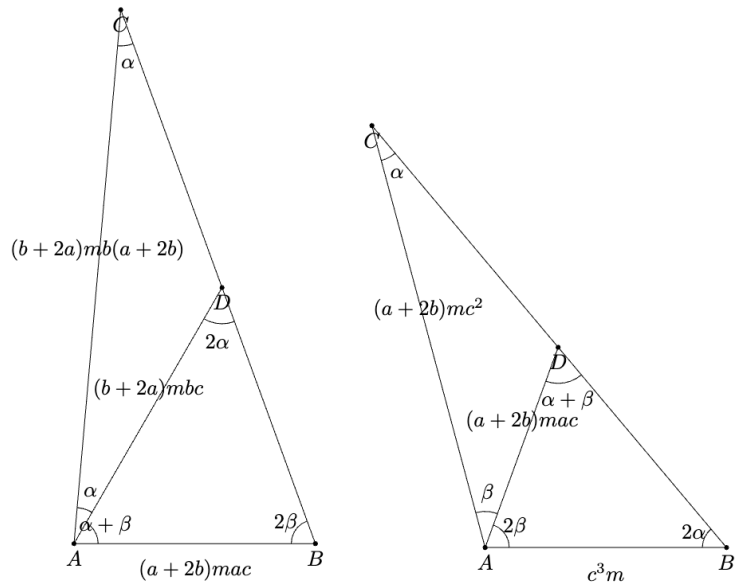

\begin{proof}
We will just do the first case; the second case is symmetric. As in Figure~\ref{fig:odd-construction}, we can tile such a triangle into a $(ma(a+b), mb(a+b), mc(a+b))$ triangle $BCD$ and a $(mab, mac, ma(a+b))$ triangle $ABD$.

Triangle $ABD$ satisfies the assumptions of Proposition~\ref{prop:arithmetic-2-construction} since $mab$ is equiconstructible. This corresponds to $m^2a(a+b)$ tiles. Triangle $BCD$ is just a $m(a+b)$-scaled copy of the $(a,b,c)$ triangle, so it offers $m^2(a+b)^2$ tiles. In total, this gives $m^2(a+b)(2a+b)$ tiles.
\end{proof}

\begin{prop}
    \label{prop:triple-angle-construction}
If the $((a+2b)am, (b+2a)bm, c^2m)$ triangle tiles into $(a,b,c)$, then
\begin{enumerate}
\item we can tile the triangle $(c^2m, (a+2b)mc, 3(a+b)mb)$ with angles $(\alpha, 2\alpha, 3\beta)$ into $3m^2(a+2b)(a+b)$ copies of $(a,b,c)$.
\item we can tile the triangle $(c^2m, (b+2a)mc, 3(a+b)ma)$ with angles $(\beta, 2\beta, 3\alpha)$ into $3m^2(2a+b)(a+b)$ copies of $(a,b,c)$.    
\end{enumerate}
\end{prop}
\begin{figure}
    \centering
    
\includegraphics[scale=0.3]{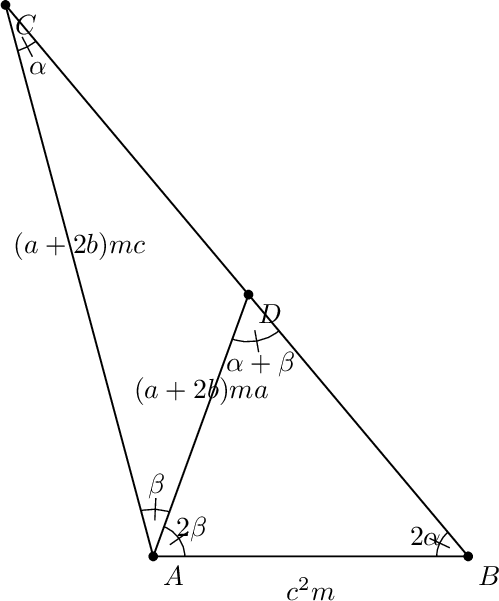}

\caption{Triangle with angles $(\alpha, 2\alpha, 3\beta)$.}
    \label{fig:triple-angle}
\end{figure}

\begin{proof}
See Figure~\ref{fig:triple-angle}. The assumptions make both triangles tileable into $(a,b,c)$. The other details are routine and similar to earlier proofs.    
\end{proof}


\begin{table}[]
    \centering
    \begin{tabular}{c|c|c}
       Angles & $N$'s via our Construction & Smallest such $N$ \\
\hline
$(\alpha+\beta, \alpha+\beta,\alpha+\beta)$ 
         & $m^2ab$, $m \geq 9$ & $1215$ (known) \\
         \hline
$(\beta, \beta,\pi-2\beta)$ 
         & $m^2a(b+2a)$, $m \geq 9$ & $2673$ (known) \\
$(\alpha, \alpha,\pi-2\alpha)$ 
         & $m^2b(a+2b)$, $m \geq 9$ & $5265$ \\
         \hline
$(\alpha, \alpha+\beta,\alpha+2\beta)$ 
         & $m^2b(a+b)$, $m \geq 9$ & $3240$ \\
$(\beta, \alpha+\beta, 2\alpha+\beta)$ 
         & $m^2a(a+b)$, $m \geq 9$ & $1944$ \\
\hline         
$(\alpha, 2\beta,2\alpha+\beta)$ 
         & $m^2(b+2a)(a+b)$, $m \geq 9$ & $7128$ \\
$(\beta, 2\alpha,\alpha+2\beta)$ 
         & $m^2(a+2b)(a+b)$, $m \geq 9$ & $8424$ \\
         \hline
$(2\beta, 2\alpha,\alpha+\beta)$ 
         & $m^2(a+2b)(b+2a)$, $m \in 5\mathbb{N}_0 + 7\mathbb{N}_0$ & $3575$ \\
         \hline
$(\alpha, 2\alpha, \pi-3\alpha)$ &$3m^2(a+2b)(a+b)$, $m \in 5\mathbb{N}_0 + 7\mathbb{N}_0$ & $7800$ \\
$(\beta, 2\beta, \pi-3\beta)$ & $3m^2(2a+b)(a+b)$, $m \in 5\mathbb{N}_0 + 7\mathbb{N}_0$ & $6600$ \\
\hline
\makecell{ $(\alpha+\beta, \alpha+\beta,\alpha+\beta)$ \\
into $(\alpha, \beta, \pi/3)$}
         & $m^2ab$, $m \geq 6$ & $1440$ (known) \\
    \end{tabular}
    \caption{Constructions for all the tilings using $(a,b,c) = (3,5,7)$ for all but the last row and $(a,b,c) = (5,8,7)$ for the last row. The constructions marked ``known'' were known to Herdt (private communication, \cite{beeson-equilateral}, and \cite{beeson-isosceles}).}
    \label{tab:incommensurable-constructions}
\end{table}

\section{Conclusion and Future Work}
\label{sec:conclusion}

Our main contribution is focusing on \textbf{ideal trapezoids} as intermediate objects in a tiling. Our constructions give conjectured solutions to half of the remaining cases of (an extension of) our motivating Erd\"{o}s problem as listed in Figure~\ref{fig:beeson-incommensurable}. To make our intuition explicit:
\begin{enumerate}
    \item We conjecture that in these cases, $T$ can be partitioned into two types of "macro-tiles": ideal trapezoids and triangles similar to $R$, as in Figures~\ref{fig:equilateral-construction-1}, \ref{fig:arithmetic-construction}, etc.
    \item We propose that the tileability of an ideal trapezoid essentially comes down to Proposition~\ref{prop:trapezoid-2}. 
\end{enumerate}

With specific tiles $(3,5,7)$ and $(5,8,7)$, we demonstrate our results in Table~\ref{tab:incommensurable-constructions}. 

The most ``obvious'' next step is to study and prove Conjecture~\ref{conj:equilateral}; all the other Conjectures come from the same root. As an intermediate step, it would be interesting (and probably required) to prove the converse to Proposition~\ref{prop:trapezoid-2}. That is,
\begin{conj}
    \label{conj:trapezoid}
    Let $ABCD$ be an ideal trapezoid and $(a,b,c)$ be as required from our setup (that is, pairwise coprime integers with $c^2 = a^2 + b^2 + ab)$. Then $ABCD$ can be tiled by $(a,b,c)$ if and only if $x > c^2-a-b$ and $(ab)|y$.
\end{conj}

\section*{Acknowledgments}
We thank Boris Alexeev, Bryce Herdt, Ariel Schreiman, and Wasin So for valuable conversation. We especially thank Michael Beeson for his generous help navigating us through the literature (including his unpublished work) and for correcting mistakes in our figures in an earlier draft.

\bibliographystyle{abbrv}
\bibliography{bibliography}

@misc{beeson-equilateral,
	author = {Michael Beeson},
	title = {Tiling an equilateral triangle},
howpublished = "\url{https://www.michaelbeeson.com/research/papers/TriangleTilingEquilateral.pdf}",
	year = {2019}}

@misc{beeson-seven,
	author = {Michael Beeson},
	title = {No triangle can be cut into seven congruent triangles},
howpublished = "\url{http://www.michaelbeeson.com/research/papers/NoSevenTiling.pdf}",
	year = {2018}}

@misc{beeson-isosceles,
	author = {Michael Beeson},
	title = {Tilings of an isosceles triangle},
	howpublished = {\url{http://www.michaelbeeson.com/research/papers/IsoscelesTilings.pdf}},
	year = {2019}}

@misc{beeson-triangletiling3,
	author = {Michael Beeson},
	title = {Triangle Tiling: the case $3\alpha + 2\beta = \pi$},
	howpublished = {\url{http://www.michaelbeeson.com/research/papers/TriangleTiling3.pdf}},
	year = {2019}}

@Inbook{soifer2009,
author="Soifer, Alexander",
title="Is There Anything Beyond the Solution?",
bookTitle="How Does One Cut a Triangle?",
year="2009",
publisher="Springer New York",
address="New York, NY",
pages="47--50",
isbn="978-0-387-74652-4",
doi="10.1007/978-0-387-74652-4_6",
url="https://doi.org/10.1007/978-0-387-74652-4_6"
}

@article{sylvester-frobenius,
 ISSN = {00029327, 10806377},
 URL = {http://www.jstor.org/stable/2369536},
 author = {J. J. Sylvester},
 journal = {American Journal of Mathematics},
 number = {1},
 pages = {79--136},
 publisher = {Johns Hopkins University Press},
 title = {On Subvariants, i.e. Semi-Invariants to Binary Quantics of an Unlimited Order},
 urldate = {2024-06-13},
 volume = {5},
 year = {1882}
}

@article{golomb1964replicating,
  title={Replicating figures in the plane},
  author={Golomb, Solomon W},
  journal={The Mathematical Gazette},
  volume={48},
  number={366},
  pages={403--412},
  year={1964},
  publisher={JSTOR}
}

@article{laczkovich2020rational,
  title={Rational Points of Some Elliptic Curves Related to the Tilings of the Equilateral Triangle},
  author={Laczkovich, Mikl{\'o}s},
  journal={Discrete \& Computational Geometry},
  volume={64},
  number={3},
  pages={985--994},
  year={2020},
  publisher={Springer}
}

@article{laczkovich1995tilings,
  title={Tilings of triangles},
  author={Laczkovich, Mikl{\'o}s},
  journal={Discrete Mathematics},
  volume={140},
  number={1-3},
  pages={79--94},
  year={1995},
  publisher={Elsevier}
}

@article{laczkovich2012tilings,
  title={Tilings of convex polygons with congruent triangles},
  author={Laczkovich, Mikl{\'o}s},
  journal={Discrete \& Computational Geometry},
  volume={48},
  number={2},
  pages={330--372},
  year={2012},
  publisher={Springer}
}

@misc{rationality,
      title={Rationality of certain triangle tilings}, 
      author={Michael Beeson and Yan X Zhang},
      year={2026},
      eprint={2604.01314},
      archivePrefix={arXiv},
      primaryClass={math.CO},
      howpublished={\url{https://arxiv.org/abs/2604.01314}}, 
}

\end{document}